\documentclass[reqno,oneside]{amsart}

\usepackage{amssymb,amsmath,amssymb,amsfonts,amsthm}
\usepackage{algorithm}
\usepackage{algorithmic}
\usepackage{bm}
\usepackage{cases}
\usepackage{cite}
\usepackage{dsfont}
\usepackage{float}
\usepackage{geometry}
\usepackage{hyperref}
\usepackage{mathtools}
\usepackage[subrefformat=parens]{subcaption}
\usepackage{tikz}
\usepackage{url}

\theoremstyle{plain}
\newtheorem{lemma}{Lemma}
\newtheorem{theorem}{Theorem}

\theoremstyle{definition}
\newtheorem{example}{Test}
\newtheorem{problem}{Problem}
\newtheorem{remark}{Remark}
\numberwithin{equation}{section}

\allowdisplaybreaks[4]

\def\BR{\mathbb R}
\def\cA{\mathcal A}
\def\cJ{\mathcal J}
\def\rd{\mathrm d}
\def\Ga{\Gamma}
\def\Om{\Omega}
\def\al{\alpha}

\def\de{\delta}
\def\ep{\epsilon}

\def\la{\lambda}

\def\om{\omega}
\def\f{\frac}
\def\nb{\nabla}
\def\pa{\partial}
\def\wt{\widetilde}

\title[Source term reconstructions in a mobile-immobile diffusion model]{Numerical reconstructions of a source term in\\
a mobile-immobile diffusion model from\\
the partial interior observation}

\author[Z. Yang]{Zhiwei Yang$^1$}
\thanks{$^1$ School of Qilu Transportation and State Key Laboratory of Intelligent Manufacturing of Advanced Construction Machinery, Shandong University, Jinan 250002, Shandong, China, e-mail: zhiweiyang@sdu.edu.cn}

\author[Y. Liu]{Yikan Liu$^2$}
\thanks{$^2$ Corresponding author. Department of Mathematics, Kyoto University, Kitashirakawa-Oiwakecho, Sakyo-ku, Kyoto 606-8502, Japan, e-mail: liu.yikan.8z@kyoto-u.ac.jp}


\begin{document}

\begin{abstract}
We consider an inverse source problem in the two-time-scale mobile-immobile fractional diffusion model from partial interior observation. Theoretically, we combine the fractional Duhamel's principle with the weak vanishing property to establish the uniqueness of this inverse problem. Numerically, we adopt an optimal control approach for determining the source term. A coupled forward-backward system of equations is derived using the first-order optimality condition. Finally, we construct a finite element conjugate gradient algorithm for the numerical inversion of the source term. Several experiments are presented to show the utility of the method.
\end{abstract}

\maketitle	


\section{Introduction}

Fractional differential equations (FDEs) have emerged as a vital tool for modeling systems with memory and hereditary properties, phenomena that classical integer-order differential equations often fail to capture adequately \cite{SunCh}. The nonlocal nature of fractional derivatives enables FDEs to describe long-range dependencies and anomalous diffusion, which are prevalent in diverse fields such as physics \cite{Bag,Ross}, biology \cite{PerKar}, finance \cite{MeeSik}, and engineering \cite{BagTor,BonKap,JaiMcKin,Macha,Mai,Pao,SunZha}. Traditional integer-order derivatives, being local operators, are limited in their ability to model such behaviors, as they cannot account for the historical influence of a system's state on its current dynamics \cite{Die}. This limitation has driven the development of fractional calculus, which provides a more versatile framework for understanding and predicting the behavior of complex systems \cite{SamKil}.

In many natural and engineered systems, dynamics evolve on multiple time scales. For instance, chemical reactions often exhibit fast and slow kinetics \cite{HinAnd}, while biological systems may involve rapid cellular responses followed by slower regulatory processes \cite{PerKar}. Classical mobile-immobile models, which separate dynamics into fast (mobile) and slow (immobile) components, have been widely used to analyze such systems \cite{SchBen,VanWie}. However, when fractional-order dynamics are introduced, the analysis becomes significantly more complex. The interplay between the nonlocal fractional derivatives and the separation of time scales poses unique challenges, necessitating novel mathematical frameworks and computational techniques \cite{Die,SamKil}. The mobile-immobile time-fractional differential equations (MIT-FDEs) thus emerge as a natural extension, enabling the modeling of systems where memory effects and multi-scale behavior coexist, such as viscoelastic materials with fast and slow relaxation processes \cite{MaiSpa,SchMet} or ecological systems with rapid population changes and long-term environmental influences \cite{KhaRaz}.

Inverse problems, particularly inverse source problems, are of paramount importance in both theoretical and applied mathematics. These problems involve the identification of unknown source terms or parameters from observed data, which is critical in many practical applications. For example, in environmental science, determining the source of pollutants from concentration measurements is essential for effective remediation strategies \cite{LugKnu}. Similarly, in medical imaging, reconstructing internal heat sources from surface temperature data is crucial for diagnosing diseases \cite{YeShi}. In the context of FDEs, inverse source problems are especially challenging due to the nonlocal nature of fractional derivatives, which often leads to ill-posedness and requires sophisticated regularization techniques for stable solutions \cite{AliAzi}. While significant progress has been made in studying inverse problems for classical FDEs \cite{JinRun}, the exploration of such problems in the context of MIT-FDEs remains largely uncharted territory.

The inverse source problem for MIT-FDEs represents a natural yet underexplored extension of this research, combining the complexities of mobile-immobile dynamics and fractional-order operators. Such problems arise in scenarios where the source term operates on multiple time scales, such as in multi-rate chemical reactions or biological systems with fast and slow regulatory mechanisms \cite{MagAbd}. Solving these problems requires not only addressing the ill-posedness inherent to inverse problems \cite{CheYua} but also accounting for the interplay between the mobile and immobile components of the system. Recent advancements in inverse source problems for fractional diffusion equations, such as the works of \cite{JLLY17} and \cite{L17}, provide valuable insights into the theoretical and numerical challenges of these problems. 

Let $T>0,q>0,\al\in(0,1)$ be constants and $\Om\subset\BR^d$ ($d=1,2,\dots$) be a bounded domain with a smooth boundary $\pa\Om$. In this paper, we are concerned with the following initial-boundary value problem for a two-scale time-fractional diffusion equation
\begin{equation}\label{main:e0}
\begin{cases}
(\pa_t u+q\,\pa_t^\al u+\cA u)(\bm x,t)=\rho(t)g(\bm x), & (\bm x,t)\in\Om\times(0,T),\\
u(\bm x,0)=a(\bm x), & \bm x\in\Om,\\
u(\bm x,t)=0, & (\bm x,t)\in\pa\Om\times(0,T).
\end{cases}
\end{equation}
Here $\pa_t^\al$ denotes the $\al$-th order Caputo derivative defined by
\begin{equation}\label{Caputo}
\pa_t^\al:=J_{0+}^{1-\al}\circ\f\rd{\rd t}:H^1(0,T)\longrightarrow L^2(0,T),
\end{equation}
where $\circ$ is the composite and $J_{0+}^{1-\al}$ stands for the $(1-\al)$-th order forward Riemann-Liouville integral operator
\[
J_{0+}^{1-\al}f(t):=\int_0^t\f{(t-s)^{-\al}}{\Ga(1-\al)}f(s)\,\rd s,\quad f\in L^2(0,T).
\]
The source term in the governing equation takes the form of separated variables, where $\rho(t)$ and $g(\bm x)$ contribute the temporal and spatial components, respectively. For simplicity, we impose the homogeneous Dirichlet boundary condition, which can be replaced by other kinds and inhomogeneous ones. Meanwhile, $\cA$ is a self-adjoint second order elliptic operator, which will be specified in detail in Section \ref{sec-theory} along with the assumptions on $\rho,g$ and the initial value $a(\bm x)$.

The governing equation in \eqref{main:e0} involves a linear combination of a first and an $\al$-th order derivatives in time. Such formulations are called mobile-immobile diffusion equations, which have been used to improve the modeling of subdiffusive transport of solutes in heterogeneous porous media \cite{SchBen}. Recently, a similar formulation is reported in Machida \cite{M25} as a diffusion approximation to the radiative transport equation with waiting time. We remark that, thanks to the integer leading order of time derivatives, it is not necessary to invoke the modern definition of $\pa_t^\al$ in fractional Sobolev spaces because the initial condition makes a pointwise sense.

This article is concerned with the following inverse problem for \eqref{main:e0}.

\begin{problem}[Inverse source problem]\label{isp}
Let $u$ be the solution to \eqref{main:e0} and $\om$ is a nonempty subdomain of $\Om$. Given the temporal component $\rho(t)$ of the source term, determine the spatial component $g(\bm x)$ the partial interior observation of $u$ in $\om\times(0,T)$.
\end{problem}

In order to invert the spatial load of the two-time-scale fractional diffusion equation, we transform the inverse problem into a minimization problem with a cost function in a least-squares form. With the help of optimal control methods, a coupled forward-backward system of differential equations is derived by utilizing the first optimality condition. Afterwards, we develop a finite element method based on the conjugate gradient algorithm to numerically invert the source term.

The remainder of the paper is organized as follows. In Section \ref{sec-theory}, we recall theoretical backgrounds and demonstrate the uniqueness of Problem \ref{isp}. After preparing a fully discrete finite element method in Section \ref{sect:num} for the direct problem, in Section \ref{sect:inverse} we adopt an optimal control approach to solve the inverse problem. Then we present several numerical experiments in Section \ref{sect:ex} to substantiate the numerical analysis results and the utility of the inverse algorithm. Finally, we present some discussion and concluding remarks in Section \ref{sect:end}.


\section{Preliminary and the theoretical uniqueness}\label{sec-theory}

For $f_1,f_2\in L^1(0,T)$, by $f_1*f_2$ we denote their convolution, that is,
\[
(f_1*f_2)(t):=\int_0^t f_1(t-s)f_2(s)\,\rd s,\quad0<t<T.
\]
Then the $(1-\al)$-th order Riemann-Liouville integral $J_{0+}^{1-\al}f$ of $f\in L^2(0,T)$ can be expressed as
\begin{equation}\label{eq-RL}
J_{0+}^{1-\al}f=k_{1-\al}*f,\quad k_{1-\al}(t):=\f{t^{-\al}}{\Ga(1-\al)}.
\end{equation}
Throughout this article, denote the inner product of $L^2(\Om)$ by $(\,\cdot\,,\,\cdot\,)$ and let $H_0^1(\Om)$, $H^2(\Om)$, $W^{1,1}(0,T)$ etc.\! be the usual Sobolev spaces. The elliptic operator $\cA:H^2(\Om)\cap H_0^1(\Om)\longrightarrow L^2(\Om)$ in \eqref{main:e0} is defined as
\[
\cA f=-\mathrm{div}(\bm A(\bm x)\nb f)+c(\bm x)f,\quad f\in H^2(\Om)\cap H_0^1(\Om),
\]
where $\bm A=(a_{ij})_{1\le i,j\le d}\in C^1(\overline\Om;\BR^{d\times d}_{\mathrm{sym}})$ is a strictly positive-definite matrix-valued function on $\overline\Om$, and $c\in L^\infty(\Om)$ is nonnegative in $\Om$. Then it is known that $\cA$ allows an eigensystem which usually provides great convenience in the solution representation. Nevertheless, in this article the eigensystem of $\cA$ does not appear explicitly.

Regarding the initial value and the source term in \eqref{main:e0}, we assume
\begin{equation}\label{eq-assume}
a\in H_0^1(\Om),\quad\rho\in L^2(0,T),\quad g\in L^2(\Om).
\end{equation}
Concerning the forward problem of \eqref{main:e0}, we state the following result.

\begin{lemma}\label{lem-forward}
The following statements hold:
\begin{enumerate}
\item Under assumption \eqref{eq-assume}, there exists a unique solution
\[
u\in L^2(0,T;H^2(\Om)\cap H_0^1(\Om))\cap H^1(0;T;L^2(\Om))
\]
to the initial-boundary value problem \eqref{main:e0}.
\item Let $g=0$ in $\Om$ and $a\in L^2(\Om)$. Then there exists a unique solution
\[
u\in L^1(0,T;H^2(\Om)\cap H_0^1(\Om))\cap W^{1,1}(0;T;L^2(\Om))
\]
to \eqref{main:e0}. Moreover, the solution $u:(0,T)\longrightarrow H^2(\Om)\cap H_0^1(\Om)$ can be analytically extended to $(0,\infty)$.
\end{enumerate}
\end{lemma}

Lemma \ref{lem-forward} (i) was established in Li, Huang and Yamamoto \cite[Theorem 2.1]{LHY21} in a slightly more general setting. For the well-posedness part of Lemma \ref{lem-forward} (ii), one can employ the explicit solution to \eqref{main:e0} and estimate the involved multinomial Mittag-Leffler functions in the same manner as in Li, Liu, and Yamamoto \cite{LLY15}. Note that although \cite{LLY15} only dealt with the formulation
\begin{equation}\label{eq-multi}
\sum_{j=1}^m q_j\pa_t^{\al_j}u+\cA u=F\quad\mbox{with }1>\al_1>\cdots>\al_m>0,
\end{equation}
the same argument works for our case with the highest order $\al_1=1$. As for the time-analyticity part of Lemma \ref{lem-forward} (ii)), Li, Huang and Yamamoto \cite[Theorem 2.3]{LHY20} established the same result in a similar framework as \eqref{eq-multi}. Again the proof is valid in our case and we omit the details here.

For the uniqueness of Problem \ref{isp}, the following two lemmas serve as essential ingredients. The first one is the following fractional Duhamel's principle connecting inhomogeneous and homogeneous problems.

\begin{lemma}\label{lem-Duhamel}
Let assumption \eqref{eq-assume} be satisfied with $a=0$. Then the unique solution to \eqref{main:e0} allows the representation
\[
u=\mu*v\quad\mbox{in }\Om\times(0,T),
\]
where $v$ solves the homogeneous problem
\begin{equation}\label{eq-IBVP-v}
\begin{cases}
\pa_t v+q\,\pa_t^\al v+\cA v=0 & \mbox{in }\Om\times(0,T),\\
v=g & \mbox{in }\Om\times\{0\},\\
v=0 & \mbox{on }\pa\Om\times(0,T),
\end{cases}
\end{equation}
and $\mu\in L^2(0,T)$ is the unique solution to the integral equation
\begin{equation}\label{eq-mu}
\mu+q\,J_{0+}^{1-\al}\mu=\rho\quad\mbox{in }(0,T).
\end{equation}
\end{lemma}

\begin{proof}
Although the above conclusion highly resembles Liu \cite[Lemma 4.2]{L17} and Jiang et al \cite[Lemma 4.1]{JLLY17}, the problem settings are different and we provide an improved proof.

First, we mention the unique existence issue of \eqref{eq-mu}. By \cite[Example 9.8]{FK78}, the operator $J_{0+}^{1-\al}:L^2(0,T)\longrightarrow L^2(0,T)$ is compact. Then according to the Fredholm alternative, it suffices to verify that $\mu=-q\,J_{0+}^{1-\al}\mu$ implies $\mu=0$ in $(0,T)$. Since
\[
|\mu|=q|J_{0+}^{1-\al}\mu|\le q\,J_{0+}^{1-\al}|\mu|,
\]
it follows immediately from a general Gr\"onwall's inequality (see Henry \cite[Lemma 7.1.1]{H81}) that $\mu=0$ in $(0,T)$. This guarantees the unique existence of a solution $\mu\in L^2(0,T)$ to \eqref{eq-mu} for any $\rho\in L^2(0,T)$.

Next, given the solution $v$ to \eqref{eq-IBVP-v} with $g\in L^2(\Om)$, we shall show that the function $\wt u:=\mu*v$ satisfies \eqref{main:e0} with the required regularity. From Lemma \ref{lem-forward} (ii), we see
\[
v\in L^1(0,T;H^2(\Om)\cap H_0^1(\Om))\cap W^{1,1}(0;T;L^2(\Om)).
\]
Then by $\mu\in L^2(0,T)$, we obtain
\[
\wt u\in L^2(0,T;H^2(\Om)\cap H_0^1(\Om))\cap H^1(0;T;L^2(\Om))
\]
and $\wt u=0$ in $\Om\times\{0\}$ in view of Young's convolution inequality. Thus $\wt u$ achieves the same regularity and the initial condition of the solution $u$ to \eqref{main:e0}.

On the other hand, it follows from direct calculation that
\[
\pa_t\wt u(t)=\pa_t\int_0^t\mu(s)v(t-s)\,\rd s=\mu(t)v(0)+\int_0^t\mu(s)\pa_t v(t-s)\,\rd s=(\mu*\pa_t v)(t)+\mu(t)\,g.
\]
Utilizing the expression \eqref{eq-RL} and the commutative law of convolutions, we calculate
\begin{align*}
\pa_t\wt u+q\,\pa_t^\al\wt u+\cA\wt u & =(\mu*\pa_t v+\mu\,g)+q\,k_{1-\al}*(\mu*\pa_t v+\mu\,g)+\cA(\mu*v)\\
& =\mu*(\pa_t v+q\,k_{1-\al}*\pa_t v+\cA v)+(\mu+q\,k_{1-\al}*\mu)g\\
& =\mu*(\pa_t v+q\,\pa_t^\al v+\cA v)+(\mu+q\,J_{0+}^{1-\al}\mu)g=\rho\,g,
\end{align*}
where we employed the governing equation of $v$ and \eqref{eq-mu}. Thus $\wt u$ also shares the same governing equation with $u$, which completes the proof.
\end{proof}

The second ingredient for Problem \ref{isp} turns out to be the following weak vanishing property.

\begin{lemma}\label{lem-WVP}
Let $v$ be the unique solution to the homogeneous problem \eqref{eq-IBVP-v} with $g\in L^2(\Om)$. Then for any nonempty subdomain $\om\subset\Om$ and any nonempty open interval $I\subset(0,T),$ $v=0$ in $\om\times I$ implies $v\equiv0$ in $\Om\times(0,T)$.
\end{lemma}

As before, a fractional counterpart of Lemma \ref{lem-WVP} was established in \cite[Theorem 2.5]{JLLY17} in the framework of \eqref{eq-multi}, whose proof relies on the time-analyticity of the solution and a Laplace transform technique. Owing to the time-analyticity by Lemma \ref{lem-forward}, the assumption $v=0$ in $\om\times I$ of Lemma \ref{lem-WVP} can be extended to $v=0$ in $\om\times(0,\infty)$, from which the remaining proof is almost identical to that for Case 2 of \cite[Theorem 2.5]{JLLY17}. Again we omit the details here.

Now we are well prepared to state the main uniqueness result for Problem \ref{isp}.

\begin{theorem}
Let $u_i$ $(i=1,2)$ be the solution to \eqref{main:e0} with
\[
g=g_i\in L^2(\Om),\quad\rho\in L^2(0,T),\quad\rho\not\equiv0\mbox{ in }(0,T),\quad a\in H_0^1(\Om).
\]
Then for any nonempty subdomain $\om\subset\Om,$ $u_1=u_2$ in $\om\times(0,T)$ implies $g_1=g_2$ in $\Om$.
\end{theorem}

\begin{proof}
Similarly to Lemma \ref{lem-Duhamel}, an almost identical conclusion was obtained in \cite[Theorem 2.6]{JLLY17} in the framework of \eqref{eq-multi} but with a stronger assumption on the known component $\rho(t)$ of the source term. Here we weaken the assumption to merely $\rho\not\equiv0$ in $(0,T)$ and renew the proof for the sake of completeness.

Putting $u:=u_1-u_2$ and $g:=g_1-g_2$, we easily see that $u$ satisfies \eqref{main:e0} with $a=0$ and $u=0$ in $\om\times(0,T)$. Then Lemma \ref{lem-Duhamel} asserts $u=\mu*v$ in $\Om\times(0,T)$ with $\mu$ and $v$ satisfying \eqref{eq-mu} and \eqref{eq-IBVP-v} respectively. Then we calculate
\[
u+q\,J_{0+}^{1-\al}u=\mu*v+q\,k_{1-\al}*(\mu*v)=(\mu+q\,k_{1-\al}*\mu)*v=\rho*v
\]
in $\Om\times(0,T)$, which, together with $u=0$ in $\om\times(0,T)$, indicates
\[
0=u+q\,J_{0+}^{1-\al}u=\rho*v\quad\mbox{in }\om\times(0,T).
\]
Picking any $\psi\in C_0^\infty(\om)$, we take inner products on both sides of the above equality with $\psi$ to derive
\[
0=(\rho*v,\psi)=\rho*v_\psi=0\mbox{ in }(0,T),\quad v_\psi(t):=(v(t),\psi).
\]
Then according to the Titchmarsh convolution theorem (see \cite[Theorem VII]{T26}), there exist constants $\tau_1,\tau_2\in[0,T]$ satisfying $\tau_1+\tau_2\ge T$ such that
\[
\rho\equiv0\mbox{ in }(0,\tau_1),\quad v_\psi\equiv0\mbox{ in }(0,\tau_2).
\]
Since it was assumed $\rho\not\equiv0$ in $(0,T)$, there should hold $\tau_1<T$ and thus $\tau_2>0$, which is independent of the choice of $\psi$. As $\psi\in C_0^\infty(\om)$ was chosen arbitrarily, we arrive at $v=0$ in $\om\times(0,\tau_2)$. Finally, we take advantage of Lemma \ref{lem-WVP} to conclude $v\equiv0$ in $\Om\times(0,\tau_2)$. Recalling that $g$ serves as the initial value of $v$, we immediately arrive at $g=0$ or equivalently $g_1=g_2$ in $\Om$.
\end{proof}


\section{Numerical scheme for solving the direct problem (\ref{main:e0})}\label{sect:num}

We introduce a numerical approximation for the two-time scale mobile-immobile subdiffusion equation described in \eqref{main:e0}. We establish a uniform division of the interval $[0,T]$ by
\[
t_n:=n\tau,\quad n=0,1,\dots,N,\quad\tau:=\f TN.
\]
Then, at time $t_n$, equation \eqref{main:e0} can be expressed as
\begin{equation}\label{PlateEqL}
\pa_t u(\bm x,t_n)+q\,\pa_t^\al u(\bm x,t_n)+\cA u(\bm x,t_n)=F(\bm x,t_n).
\end{equation}
We discretize the derivatives of $u$ with respect to time, $\pa_t u$ at $t=t_n$ by
\[
\pa_t u(\bm x,t_n)\approx\de_\tau u(\bm x,t_n):=\f{u(\bm x,t_n)-u(\bm x,t_{n-1})}\tau.
\]
We then approximate $\pa_t^\al u$ at $t=t_n$ using a classical L1 scheme (see \cite{LinXu,Sun,SunGao,SunWu}) as
\begin{align*}
\pa_t^\al u(\bm x,t_n) & \approx\de^\al_\tau u(\bm x,t_n):=\f1{\Ga(1-\al)}\sum_{k=1}^n\int_{t_{k-1}}^{t_k}\f{\de_\tau u(\bm x,t_k)}{(t_n-t)^\al}\,\rd t\\
& =\f1{\Ga(2-\al)}\sum_{k=1}^n\left[(t_n-t_{k-1})^{1-\al}-(t_n-t_k)^{1-\al}\right]\de_\tau u(\bm x,t_k)\\
& = \sum_{k=1}^n c_{n,k}\,\de_\tau u(\bm x,t_k)
\end{align*}
with
\[
c_{n,k}:=\f1{\Ga(2-\al)}\left[(t_n-t_{k-1})^{1-\al}-(t_n-t_k)^{1-\al}\right].
\]

Let $\chi=\chi(\bm x)$ be any test function possessing square integrable first-order derivatives and adhering to the boundary condition \eqref{main:e0}. We proceed by multiplying equation \eqref{PlateEqL} by $\chi$, integrating the equation over $\Om$, performing integration by parts, and utilizing the aforementioned approximations to derive the weak formulation
\begin{align*}
& \int_\Om\left\{\left(\de_\tau u(\bm x,t_n)+q\,\de_\tau^\al u(\bm x,t_n+c(\bm x)\right)\chi(\bm x)+\bm A(\bm x)\nb u(\bm x,t_n)\cdot\nb\chi(\bm x)\right\}\rd\bm x\\
& =\int_\Om F(\bm x,t_n)\chi(\bm x)\,\rd\bm x.
\end{align*}

Let $S_H(\Om)$ denote the finite element space comprised of continuously differentiable and piecewise  linear polynomials defined $\Om$. The procedure of the  finite element method is as follows: For $n=1,2,\ldots,N$, determine $u_{h,n}(\bm x)\in S_H(\Om)$ such that for $\chi_h\in S_H(\Om)$,
\begin{align*}
& \int_\Om\left\{\left(\de_\tau u_{h,n}(\bm x)+q\,\de_\tau^\al u_{h,n}(\bm x)+c(\bm x)\right)\chi_h+\bm A(\bm x)\nb u_{h,n}(\bm x)\cdot\nb\chi_h\right\}\rd\bm x\\
& =\int_\Om F(\bm x,t_n)\chi_h\,\rd\bm x
\end{align*}


\section{Inverse problem and its conjugate gradient algorithm}\label{sect:inverse}


\subsection{Mathematical formulation of the inverse problem and the corresponding inverse algorithm}

Within this section, we transform the inverse source problem into a minimization problem using the framework of optimal control. Now, we present the optimal control formulation for the regularized inverse source problem. We define the regularized form of the Tikhonov functional by extending it as follows:
\begin{align}
\cJ(g) & :=\f1{2}\int_0^T\!\!\!\int_\om\left(u(\bm x,t;g)-u_d(\bm x,t)\right)^2\,\rd\bm x\rd t+\f{\beta}{2}\int_\Om g^2(\bm x)\,\rd\bm x,\label{ocp:e0}\\
\mathcal U_{\mathrm{ad}} & =\{g\in L^2(\Om):|g|\le g_{\max}\mbox{ a.e.\! in }\Om\},\nonumber
\end{align}
where $u_d(\bm x, t)$ represents the data or measured solution in the interior observation subdomain $\omega$ over the time interval $(0, T)$, $u(\bm x,t;g)$ is the solution of the problem defined in \eqref{main:e0}, and $\beta>0$ is the parameter of  regularization. Subsequently, the inverse problem is reformulated as an optimal control problem: determine $g\in\mathcal U_{\mathrm{ad}}$ that minimizes the functional $\cJ(g)$ under the constraint that $u\left(\bm x,t;g\right)$ satisfies the model problem \eqref{main:e0}.

Now, we prove the uniqueness of the inverse source problem for the initial-boundary value mobile-iimmobile diffusion equation

\begin{theorem}\label{thm2}
The derivative of the objective function $\cJ$ defined in \eqref{ocp:e0} with respect to the parameter $g$ can be expressed as follows:
\begin{equation}\label{thmIP:e0}
\cJ'(g)=\int_0^T\rho(t)v(\bm x,t)\,\rd t+\beta g(\bm x),
\end{equation}
where $v(\bm x,t)$ solves the adjoint problem
\begin{equation}\label{adj:e1}
\begin{cases}
-\pa_t v(\bm x,t)+q\,\hat\pa_t^\al v(\bm x,t)+\cA v(\bm x,t)=\mathds{1}_\om(u(\bm x,t;g)-u_d(\bm x,t)), & (\bm x,t)\in\Om \times[0,T),\\
v(\bm x,T)=0, & \bm x\in\Om,\\
v(\bm x,t)=0, & (\bm x,t)\in\pa\Om\times [0,T).
\end{cases}
\end{equation}
Here $\mathds{1}_\om$ denotes the characteristic function of $\om$ and $\hat\pa_t^\al$ denotes the backward Riemann-Liouville differential operator defined by (see \cite{Pod})
\begin{equation}\label{RiemannD}
\hat\pa_t^\al z:=-\pa_t\,({}_t\hat I_T^{1-\al}z),\quad{}_t\hat I_T^{1-\al}:=\f1{\Ga(1-\al)}\int_t^T\f{z(s)}{(s-t)^\al}\,\rd s.
\end{equation}
\end{theorem}

\begin{proof}
In order to calculate the gradient of the  cost function with respect to $g$, we compute
\begin{align}
\de\cJ(g) & =\cJ(g+\de g)-\cJ(g)\nonumber\\
&= \int_0^T\!\!\!\int_\om(u(\bm x,t;g)-u_d(\bm x,t))\,\de u\,\rd\bm x\rd t+\beta\int_\Om g\,\de g\,\rd\bm x + o(\de g)\label{thmIP:e1}
\end{align}
as $\|\de g\|_{L^2(\Om)}\to0$, where $\de u:=u(\bm x,t;g+\de g)-u(\bm x,t;g)$ solves the following sensitivity system
\begin{equation}\label{thmIP:e2}
\begin{cases}
\pa_{t}\de u(\bm x,t)+q\,\pa_t^\al\de u(\bm x,t)+\cA\de u(\bm x,t)=\rho(t)\de g(\bm x), & (\bm x,t)\in\Om\times(0,T],\\
\de u(\bm x,0)=0, & \bm x\in\Om,\\
\de u(\bm x,t)=0, & (\bm x,t)\in\pa\Om\times(0,T].
\end{cases}
\end{equation}
Then we multiply $v(\bm x,t)$ on both sides of the first equation in \eqref{thmIP:e2} and integrate over $\Om\times(0,T)$ to obtain
\begin{equation}\label{thmIP:e3}
\int_0^T\!\!\!\int_\Om\rho(t)\de g(\bm x)v(\bm x,t)\,\rd\bm x\rd t=\int_0^T\!\!\!\int_\Om(\pa_t\de u+q\,\pa_t^\al \de u+\cA\de u)\,v\,\rd\bm x\rd t.\\
\end{equation}
For the first term on the right-handed side of \eqref{thmIP:e3}, we employ the initial condition of $\de u$ in \eqref{thmIP:e2} and the terminal condition of $v$ in \eqref{adj:e1} to derive
\begin{align}
\int_0^T\!\!\!\int_\Om v\,\pa_t\de u\,\rd\bm x\rd t & =\int_0^T\!\!\!\int_\Om(-\pa_t v)\,\de u\,\rd\bm x\rd t+\int_\Om\big[v\,\de u\big]_{t=0}^{t=T}\,\rd\bm x\nonumber\\
& =\int_0^T\!\!\!\int_\Om(-\pa_t v)\,\de u\,\rd\bm x\rd t.\label{thmIP:e4}
\end{align}
Next, we use the forward Caputo fractional operator $\pa_t^\al$ in \eqref{Caputo} and the backward differential operator $\hat\pa_t^\al$ defined in \eqref{RiemannD} to integrate $v\,\pa_t^\al\de u$ in the second term on the right-handed side of \eqref{thmIP:e3} to find
\begin{align}
\int_\om\int_0^T v\,\pa_t^\al\de u\,\rd t\rd\bm x & = \f1{\Ga(1-\al)}\int_\om\int_0^T v(\bm x,t)\int_0^t\f{\pa_s\de u(\bm x,s)}{(t-s)^\al}\,\rd s\rd t\rd\bm x\nonumber\\
& =\f1{\Ga(1-\al)}\int_\om\int_0^T\!\!\!\int_0^t\f{v(\bm x,t)\pa_s\de u(\bm x,s)}{(t-s)^\al}\,\rd s\rd t\rd\bm x\nonumber\\
& =\f1{\Ga(1-\al)}\int_\om\int_0^T\!\!\!\int_s^T\f{v(\bm x,t)\pa_s\de u(\bm x,s)}{(t-s)^\al}\,\rd t\rd s\rd\bm x\nonumber\\
& =\int_\om\int_0^T\pa_s\de u(\bm x,s)\f1{\Ga(1-\al)}\int_s^T\f{v(\bm x,t)}{(t-s)^\al}\,\rd t\rd s\rd\bm x\nonumber\\
& =\int_\om\int_0^T\de u(\bm x,s)\hat\pa_s^\al v(\bm x,s)\,\rd s\rd\bm x-\int_\om\de u(\bm x,0){}_0\hat I_T^{1-\al}v(\bm x,s)\,\rd\bm x\nonumber\\
& =\int_0^T\!\!\!\int_\om\de u(\bm x,t)\hat\pa_t^\al v(\bm x,t)\,\rd\bm x\rd t.
\end{align}
For the third term on the right-hand side of \eqref{thmIP:e3}, we apply the Green formula and take into account the boundary conditions \eqref{adj:e1} for $v$ and \eqref{thmIP:e2} for $\de u$ to obtain an equivalent expression
\begin{equation}\label{thmIP:e6}
\int_0^T\!\!\!\int_\Om(\cA\de u)v\,\rd\bm x\rd t=\int_0^T\!\!\!\int_\Om\cA v\,\de u\,\rd\bm x\rd t.
\end{equation}
Since $v(\bm x,t)$ is the solution of the adjoint problem \eqref{adj:e1}, then we substitute equations \eqref{thmIP:e4}--\eqref{thmIP:e6} into \eqref{thmIP:e3} to arrive at
\begin{align*}
\int_0^T\!\!\!\int_\Om\rho(t)\de g(\bm x)v(\bm x,t)\,\rd\bm x\rd t & =\int_0^T\!\!\!\int_\Om\left(\pa_t\de u+q\,\pa_t^\al\de u+\cA\de u\right)v\,\rd\bm x\rd t\\
& = \int_0^T\!\!\!\int_\Om\left(-\pa_t v+q\,\hat\pa_t^\al v+\cA v\right)\de u\,\rd\bm x\rd t\\
& = \int_0^T\!\!\!\int_\Om\mathds{1}_\om(u(\bm x,t;g)-u_d))\de u\,\rd\bm x\rd t\\
& = \int_0^T\!\!\!\int_\om(u(\bm x,t;g)-u_d))\de u\,\rd\bm x\rd t.
\end{align*}
Finally, we substitute the above equality into \eqref{thmIP:e1} to arrive at
\[
\de J(g)=\int_0^T\!\!\!\int_\Om\rho(t)\de g(\bm x)v(\bm x,t)\,\rd\bm x\rd t+\beta\int_\Om g(\bm x)\de g(\bm x)\,\rd\bm x + o(\de g)
\]
as $\|\de g\|_{L^2(\Om)}\to0$. Consequently, we can infer that
\[
J'(g)=\int_0^T\rho(t)v(\bm x,t)\,\rd t+\beta g(\bm x),
\]
which completes the proof.
\end{proof}

The above theorem provides valuable information on how changes in the parameter $g$ affect the overall cost function, guiding us towards an optimal solution.


\subsection{The reconstruction algorithm: finite element conjugate gradient algorithm}

In this subsection, we are now in a position to illustrate the conjugate gradient method to recover the spatial load $g(\bm x)$ from partial interior observation.

Upon examination of the demonstration provided in Theorem \ref{thm2}, it becomes apparent that the computation of the gradient of the cost function $\cJ$ with respect to $g$ relies on the resolution of coupled forward-backward systems. These systems involve not only addressing the direct problem \eqref{main:e0} but also tackling the adjoint problem \eqref{adj:e1}. Here, we need to emphasize: due to the structural similarity between the adjoint problem and the forward problem, the finite element method and error analysis aforementioned are equally applicable to solving the adjoint problem \eqref{adj:e1}.

Now, we introduce and elaborate on the implementation of the conjugate gradient algorithm (Algorithm \ref{alg_inv}) for the inversion of the source term.  Through this algorithmic approach, we aim to effectively reconstruct the source term by iteratively approaching the solution until convergence is achieved. This method offers a systematic and reliable framework for optimizing the inversion process and ultimately enhancing the overall performance of our computational model. Under our problem's formulation, $J(k)$ is convex. For convex optimization problems, any stationary point (where $J'(k)=0$) is indeed a global minimum. Therefore, the stopping criterion $\|\mathcal{J}'(g^k)\|_{L^2(\Omega)}\le\mathrm{Tolerance}$ is justified for locating the global optimum.

\begin{algorithm}[!h]
\caption{Finite Element Conjugate Gradient Algorithm}
\begin{algorithmic}
\STATE Initialize $k = 0$, and choose the initial guess $q^k$ with given initial data, boundary conditions, and the partial observation measurement data $u_d(\bm x,t)\in\om\times[0,T]$.
\FOR{$k=1,\dots,K$}
\STATE Solve the forward problem corresponding to the given data $g^k$ to obtain $u(\bm x,t;g^k)$.
\STATE Solve the adjoint problem \eqref{adj:e1} using the partial interior observation data.
\STATE Compute the gradient $\cJ'(g^k)$ using formula \eqref{thmIP:e0}.
\STATE Calculate the descent direction $d^k = -\cJ'(g^k)$.
\STATE Determine the search step $\la^k$ for the minimization problem using the Armijo rule:
$$\cJ(g^k+\la^k d^k)=\min_{\la>0}\cJ(g^k+\la d^k).$$
\STATE Update $g^{k+1}=g^k-\la^k\cJ'(g^k)$.
\STATE If $\|\cJ'(g^k)\|_{L^2(\Om)}\le\mathrm{Tolerance}$, then stop. Otherwise, continue looping.
\ENDFOR
\end{algorithmic}
\label{alg_inv}
\end{algorithm}

\begin{remark}
The selection of an appropriate regularization parameter $\beta$ is a critical step in solving ill-posed inverse problems, as it governs the trade-off between solution fidelity and numerical stability. This challenge has been extensively studied, yielding several principled methodologies. Several approaches include the L-curve criterion \cite{Hansen1992}, which identifies $\beta$ at the corner of the solution-residual norm plot; the discrepancy principle \cite{Morozov1984}, which sets $\beta$ such that the residual norm matches the estimated noise level; and the Generalized Cross-Validation (GCV) method \cite{Golub1979}, which determines $\beta$ by minimizing the predicted squared error. These methods, comprehensively discussed in \cite{Vogel2002}, provide formal frameworks for parameter selection.

In this work, we employ a pragmatic heuristic inspired by these established principles. Our approach selects $\beta$ by balancing the orders of magnitude between the estimated observation error and the regularization term. This strategy ensures a practical compromise between accuracy and stability: when observational uncertainty is significant, a larger $\beta$ value provides stronger regularization to suppress noise effects and prevent overfitting; conversely, with high-quality data, a smaller $\beta$ permits closer adherence to observations. While less automated than formal criteria, this scale-balancing method offers a computationally efficient and intuitively clear alternative that effectively captures the essential compromise embodied in more sophisticated techniques, making it particularly suitable for the class of problems investigated in this study.
\end{remark}


\section{Numerical experiments}\label{sect:ex}

In this section, we will conduct several numerical experiments to validate the accuracy of our finite element algorithm in solving inverse problems and to assess the efficiency of our algorithm in addressing inverse problems. As for inverse problems, we will evaluate the algorithm's robustness against noise interference and parameter errors, and investigate the accuracy and reliability of parameter estimation. These experiments will provide valuable insights into the performance and applicability of the algorithm, guiding us in further optimizing the algorithm for practical engineering applications. All the computations were implemented using MATLAB R2021b on a ThinkPad X1 Carbon (2021) Laptop with an Intel Core i7 (2.80 GHz) CPU and 16.0 GB RAM.


\subsection{Numerical experiments of the finite element conjugate gradient method}\label{subsect:ex2}

In the following subsections, we consider the contaminated data which the observation were noised as following
\begin{equation}\label{noise_free}
\tilde u_d(\bm x,t)=u_d(\bm x,t)+\ep\%\cdot\mathrm{rand}(-1,1),\quad(\bm x,t)\in\om\times[0,T].
\end{equation}
Here, $\ep$ represents the noise level, and rand denotes the random number uniformly distributed in $[-1,1]$. We calculate the relative $L^2$ error by
\[
\mathrm{Error}:=\f{\|g_{\rm true}-g_{\rm reconstruction}\|_{L^2(\Om)}}{\|g_{\rm true}\|_{L^2(\Om)}}.
\]

\begin{example}
Without loss of generality, we assume that the domain $\Om=(0,1)^2$ and the parameter $q=1$. The illustration of the domain is presented in Figure \ref{figure:gvisco}, where $\om$ is the observation domain. Further, we fix
\[
\rho(t)=2+ (2\pi t)^2,\quad g_{\rm true}(x,y)=\f1{2}\cos(\pi x)\cos(\pi y)+1.
\]

\begin{figure}[htbp]\centering
\includegraphics[scale=.4]{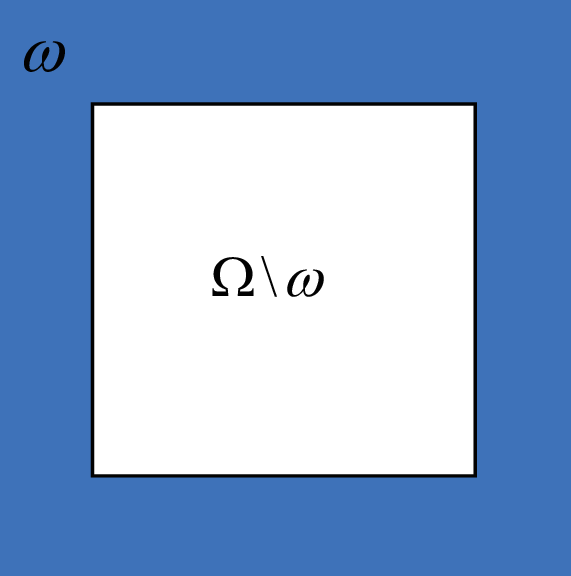}\hspace{0.1in}
\caption{Illustration of the domain.}\label{figure:gvisco}
\end{figure}

\begin{figure}[htbp]\centering
\includegraphics[width=6.5in,height=5in]{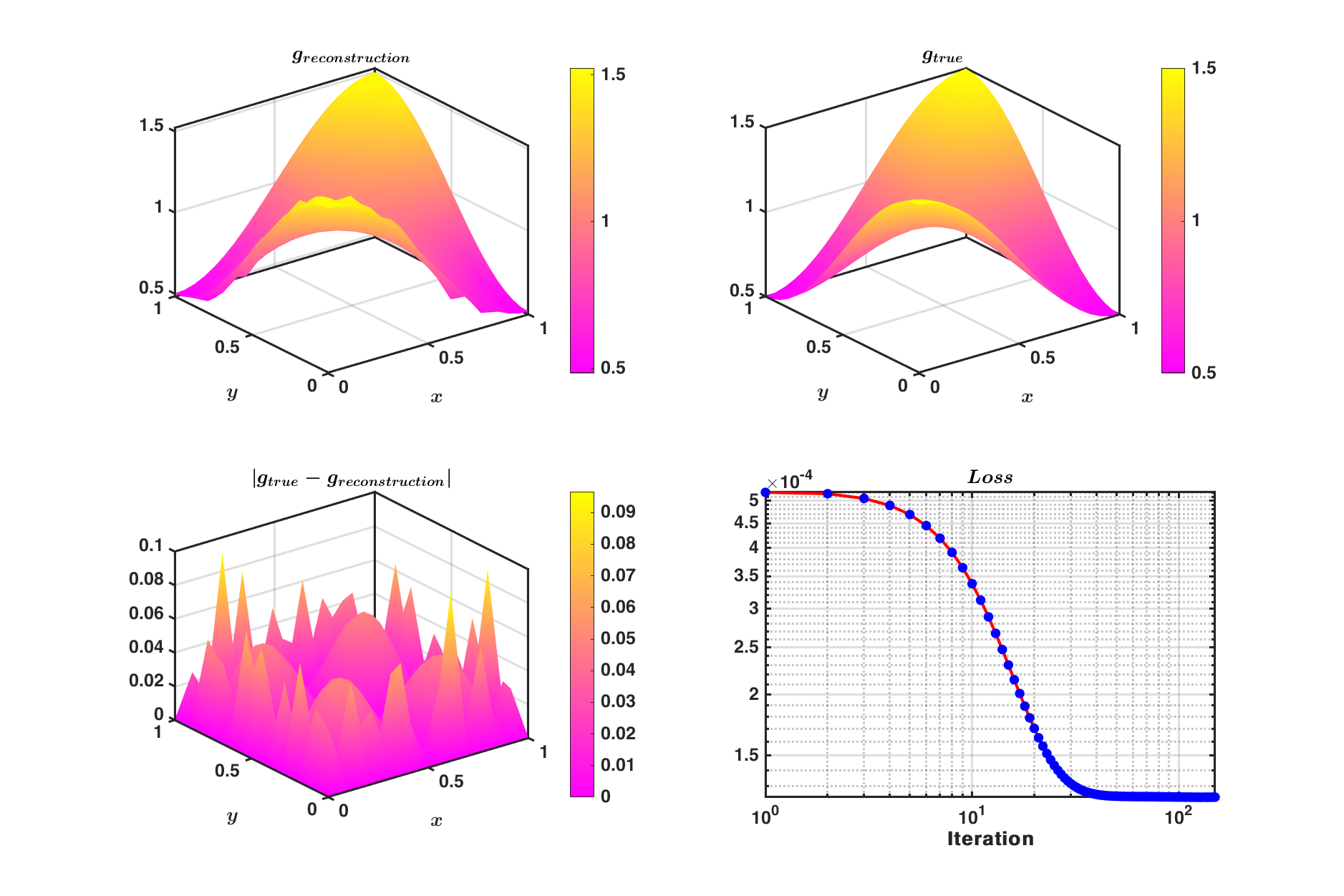}\hspace{0.15in}
\caption{We fix the order $\al=0.5$ and present the numerical reconstruction for the inverse source of the membrane. The observable subdomain $\om=\Om\setminus[0.05,0.95]^2$, $T=1.5$ and the noise level $\ep=1$ in \eqref{noise_free}. The first row from left to right are the reconstructed and the true solutions, respectively. The second row represents the absolute error and the corresponding loss with respect to the iterations.}\label{figure:3}
\end{figure}

We discretize the space-time region using a $20^2\times20$ mesh for computation. The numerical reconstructions are illustrated in Figure \ref{figure:3}. Despite the presence of noise in the observations, the loss function decreases to $10^{-4}$ within a few iterations, demonstrating the effectiveness of the algorithm. Additionally, we evaluate the algorithm by varying the subdomains and noise levels. For instance, we examine the impact of setting the noise disturbance to 1\%, 3\%, or 5\% while keeping the observation range fixed, and also assess the algorithm by altering the observation range while maintaining a fixed noise level. The results of these tests are summarized in Table \ref{table:1}.

\begin{table}[htbp]\centering
\setlength{\tabcolsep}{9mm}
\caption{We fix $\al=0.5$, and  present numerical reconstructions under the different choices of the noise levels $\ep$ and observation subdomains $\om$.}\vspace{0.5em}
{\begin{tabular}{cccc}
\hline$\ep\ (\%)$ & $\om$ & Error & Loss\\
\hline
1 & $\Om\setminus[0.1,0.9]^2$ & $2.43\times10^{-2}$ & $1.26\times10^{-4}$\\
3 & $\Om\setminus[0.1,0.9]^2$ & $6.34\times10^{-2}$ & $1.90\times10^{-4}$\\
5 & $\Om\setminus[0.1,0.9]^2$ & $9.02\times10^{-2}$ & $3.18\times10^{-4}$\\
\hline
1 & $\Om\setminus[0.2,0.8]^2$ & $2.68\times10^{-2}$ & $1.29\times10^{-4}$\\
1 & $\Om\setminus[0.1,0.9]^2$ & $2.43\times10^{-2}$ & $1.26\times10^{-4}$\\
1 & $\Om\setminus[0.05,0.95]^2$ & $2.91\times10^{-2}$ & $1.23\times10^{-4}$\\
\hline
\end{tabular}}\label{table:1}
\end{table}

In order to test the influence of the fractional orders, we also present different choices of the different fractional orders $\al=0.3,\ 0.6,\ 0.9$. The numerical results are listed in Table \ref{table:2}. As expected, the finite element conjugate gradient method demonstrates the robustness against oscillating noises in the observation data.

\begin{table}[htbp]\centering
\setlength{\tabcolsep}{9mm}
\caption{We fix noise level $\ep=2$ and observation subdomains $\om=\Om\setminus[0.05,0.95]^2$, and present numerical reconstructions under the different choices of the fractional order $\al$.}\vspace{0.5em}
{\begin{tabular}{cccc}
\hline
$\al$ & $\om$ & Error & Loss\\
\hline
$0.3$ & $\Om\setminus[0.05,0.95]^2$ & $4.60\times10^{-2}$ & $1.41\times10^{-4}$\\
$0.6$ & $\Om\setminus[0.05,0.95]^2$ & $4.15\times10^{-2}$ & $1.40\times10^{-4}$\\
$0.9$ & $\Om\setminus[0.05,0.95]^2$ & $4.63\times10^{-2}$ & $1.40\times10^{-4}$\\
\hline
\end{tabular}}\label{table:2}
\end{table}
\end{example}

\begin{example}
In this test, we investigate the influence of the monotonicity of $g_{\rm true}(x,y)$  upon the numerical inverse problem. We select three different true solutions of $g_{\rm true}(x,y)$ as follows,
\begin{align*}
g_{\rm true}(x,y) & =3-\exp\left(1-\f{x+y}{2}\right),\\
g_{\rm true}(x,y) & =\f1{2}\cos(\pi x)\cos(2\pi y)+1,\\
g_{\rm true}(x,y) & =\f1{2}\sin(\pi x)\cos(\pi y)+1.\\
\end{align*}
Here, we take the observation domain $\om=\Om\setminus[0.05,0.95]^2$, $\al=0.5$, and the noise level $\ep=1$. Other data is the same as in \textbf{Test 1}. We show the numerical reconstructions and true solutions in Figure \ref{figure:4}. We also list the relative errors and the corresponding loss in Table \ref{table:3}.
\begin{figure}[htbp]\centering
\includegraphics[width=5.5in,height=2.5in]{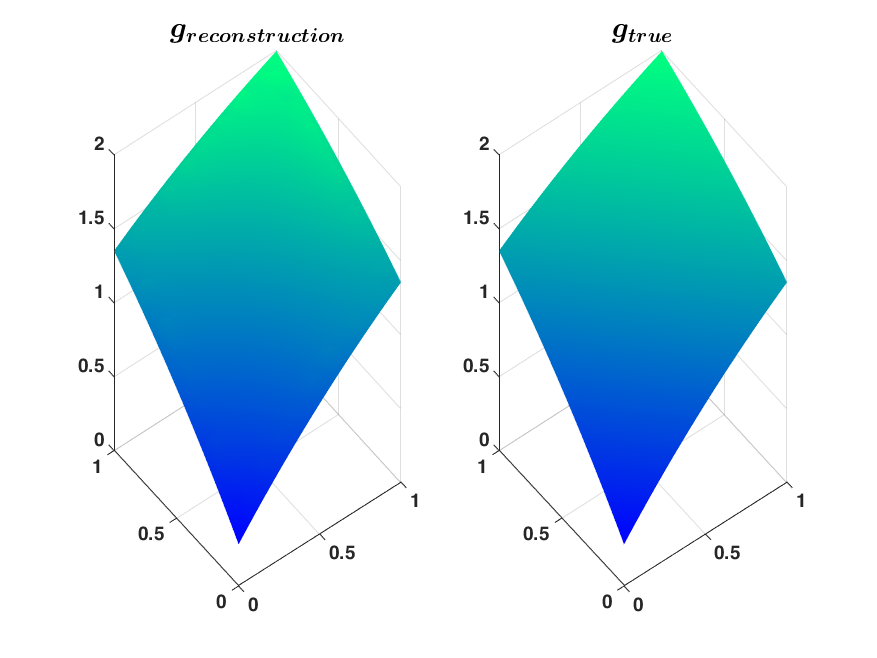}\hspace{0.15in}
\includegraphics[width=5.5in,height=2.5in]{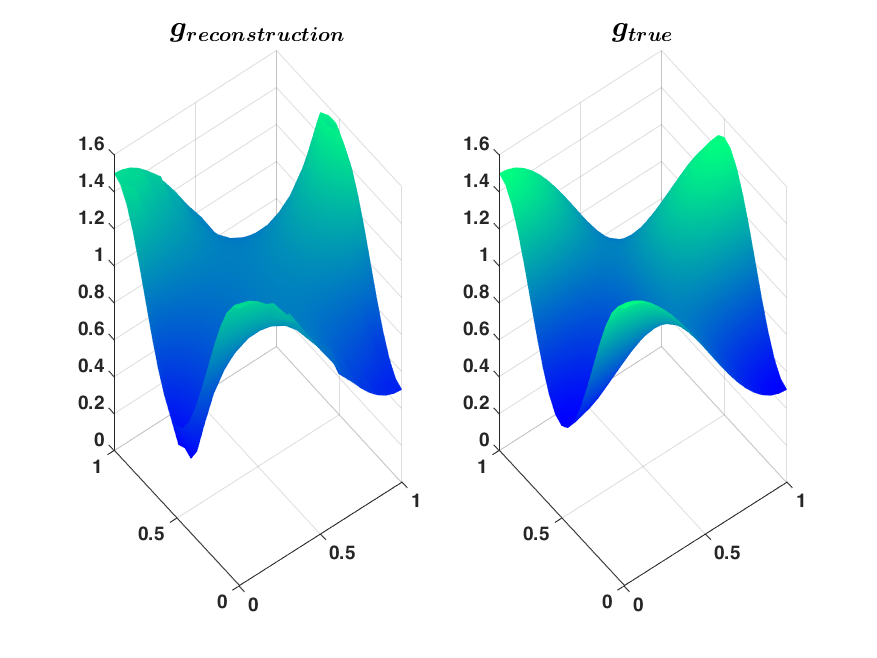}\hspace{0.15in}
\includegraphics[width=5.5in,height=2.5in]{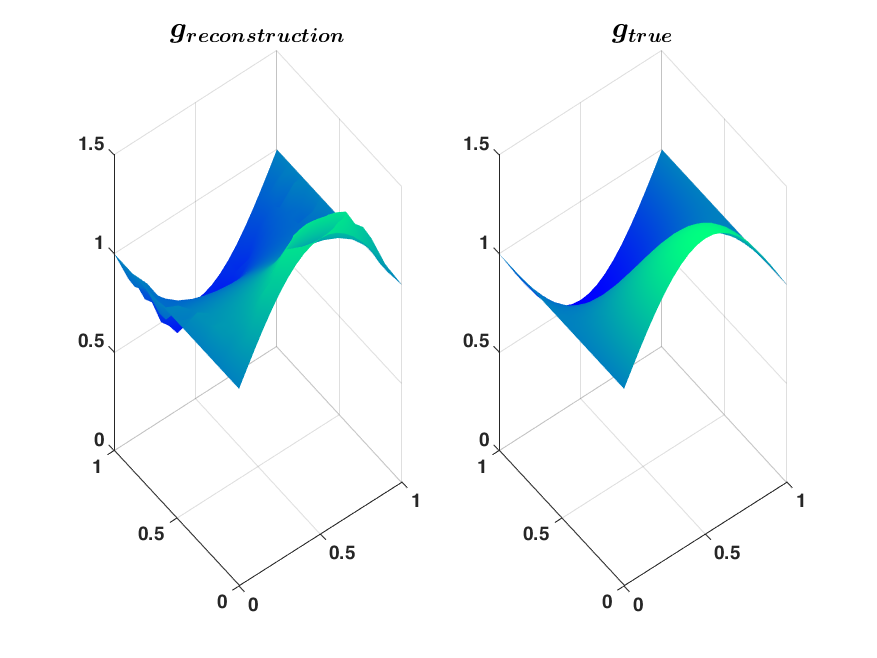}\hspace{0.15in}
\caption{We fix the order $\al=0.5$ and present the numerical reconstruction for the inverse source of the membrane. The observable subdomain $\om=\Om\setminus[0.05,0.95]^2$, $T=1.5$ and the noise level $\ep=1$ in \eqref{noise_free}. The first row to the third row from left to right are the reconstructed and true solutions, respectively.}\label{figure:4}
\end{figure}

\begin{table}[htbp]\centering
\setlength{\tabcolsep}{9mm}
\caption{We fix noise level $\ep=1$ and observation subdomains $\om=\Om\setminus[0.05,0.95]^2$, and present numerical reconstructions under the different choices of the function $g_{\rm true}(x,y)$.}\vspace{0.5em}
{\begin{tabular}{cccc}
\hline
$g_{\rm true}(x,y)$ & Error & Loss\\
\hline
$3-\exp(1-\f{x+y}2)$ & $2.75\times10^{-2}$ & $2.09\times10^{-4}$\\
$\f{\cos(\pi x)\cos(2\pi y)}2+1$ & $6.61\times10^{-2}$ & $1.23\times10^{-4}$\\
$\f{\sin(\pi x)\cos(\pi y)}2+1$ & $4.78\times10^{-2}$ & $1.22\times10^{-4}$\\
\hline
\end{tabular}}\label{table:3}
\end{table}
\end{example}


\section{Concluding remarks}\label{sect:end}

In this paper, we consider the inverse source problem for a mobile-immobile diffusion model from the partial interior observation via the optimal control method. In particular, we reformulate the inverse source problem as a minimization problem, then we derive its optimal control condition and develop a finite element conjugate gradient algorithm. We finally give several numerical experiments to show the utility and efficiency of the inverse algorithm. Furthermore, the current framework shows potential for extension to more complex scenarios, such as variable-order fractional mobile-immobile diffusion models \cite{ZhengWang2021,ZhengWang2021b} or variable-exponent subdiffusion problems based on model transformation techniques \cite{ZhengLiQiu2024,Zheng2025}.
Meanwhile, another possible extension in the formulation could be the generalization of the coefficient $q$ of the fractional derivative $\pa_t^\al$ to $q(\bm x)$ or even $q(\bm x,t)$. In the case of $q(\bm x)$, the fractional Duhamel's principle and the weak vanishing property still hold and thus the argument in this article is expected to work, while the case of $q(\bm x,t)$ should definitely be more challenging. For both cases, the numerical framework in this article could potentially be inherited. Similarly, identifying a general source term $F(\bm x,t)$ without the structure of separated variables also awaits further consideration.
These extensions would allow for more flexible characterization of heterogeneous media with spatially or temporally varying memory effects as well as varieties in the source term, representing promising directions for future investigation.


\section*{Acknowledgements}

The first author is supported by the National Natural Science Foundation of China (No. 12401555). The second author is supported by JSPS KAKENHI Grant Numbers JP22K13954, JP23KK0049, Guangdong Basic and Applied Basic Research Foundation (No. 2025A1515012248) and FY2025 MUSUBIME of Kyoto Universit.



\begin{thebibliography}{00}

\bibitem{AliAzi}
M. Ali, S. Aziz and S. A. Malik, Inverse source problems for a space-time fractional differential equation, {\it Inverse Probl. Sci. Eng.} {\bf28} (2020), no. 1, 47--68.

\bibitem{Bag}
R. L. Bagley, Power law and fractional calculus model of viscoelasticity, {\it AIAA J.} {\bf27} (1989), no. 10, 1412--1417.

\bibitem{BagTor}
R. L. Bagley and P. J. Torvik, A theoretical basis for the application of fractional calculus to viscoelasticity, {\it J. Rheology} {\bf27} (1983), no. 3, 201--210.

\bibitem{BonKap}
A. Bonfanti, J. L. Kaplan, G. Charras and A. Kabla, Fractional viscoelastic models for power-law materials, {\it Soft Matter} {\bf16} (2020), 6002--6020.

\bibitem{CheYua}
X. Cheng, L. Yuan and K. Liang, Inverse source problem for a distributed-order time fractional diffusion equation, {\it J. Inverse Ill-Posed Probl.} {\bf28} (2020), no. 1, 17--32.

\bibitem{Die}
K. Diethelm, {\it The Analysis of Fractional Differential Equations}, Lecture Notes in Math. 2004, Springer, Berlin, 2010.

\bibitem{FK78}
H. Fujita and S. T. Kuroda, {\it Functional Analysis}, Iwanami Shoten, Tokyo, 1978.

\bibitem{Golub1979}
G. H. Golub, M. Heath and G. Wahba, Generalized cross-validation as a method for choosing a good ridge parameter, {\it Technometrics} {\bf21} (1979), no. 2, 215--223.

\bibitem{Hansen1992}
P. C. Hansen, Analysis of discrete ill-posed problems by means of the L-curve, {\it SIAM Rev.} {\bf34} (1992), no. 4, 561--580.

\bibitem{H81}
D. Henry, {\it Geometric Theory of Semilinear Parabolic Equations}, Lecture Notes in Math. 840, Springer, Berlin, 1981.

\bibitem{HinAnd}
M. Hinze, A. Schmidt and R. I. Leine, The direct method of Lyapunov for nonlinear dynamical systems with fractional damping, {\it Nonlinear Dyn.} {\bf102} (2020), 2017--2037.

\bibitem{JaiMcKin}
A. Jaishankar and G. H. McKinley, Power-law rheology in the bulk and at the interface: Quasi-properties and fractional constitutive equations, {\it Proc. R. Soc. Lond. Ser. A Math. Phys. Eng. Sci.} {\bf469} (2013), no. 2149, Article ID 20120284.

\bibitem{JLLY17}
D. Jiang, Z. Li, Y. Liu and M. Yamamoto, Weak unique continuation property and a related inverse source problem for time-fractional diffusion-advection equations, {\it Inverse Problems} {\bf33} (2017), no. 5, Article ID 055013.

\bibitem{JinRun}
B. Jin and W. Rundell, An inverse problem for a one-dimensional time-fractional diffusion problem, {\it Inverse Problems} {\bf28} (2012), no. 7, Article ID 075010.

\bibitem{KhaRaz}
N. A. Khan, O. A. Razzaq, S. P. Mondal and Q. Rubbab, Fractional order ecological system for complexities of interacting species with harvesting threshold in imprecise environment, {\it Adv. Difference Equ.} {\bf2019} (2019), Paper No. 405.

\bibitem{LHY20}
Z. Li, X. Huang and M. Yamamoto, Initial-boundary value problems for multi-term time-fractional diffusion equations with $x$-dependent coefficients, {\it Evol. Equ. Control Theory} {\bf9} (2020), no. 1, 153--179.

\bibitem{LHY21}
Z. Li, X. Huang and M. Yamamoto, Well-posedness and asymptotic estimate for a diffusion equation with time-fractional derivative, {\it Chinese Ann. Math. Ser. B} {\bf46} (2025), no. 1, 115--138.

\bibitem{LLY15}
Z. Li, Y. Liu and M. Yamamoto, Initial-boundary value problems for multi-term time-fractional diffusion equations with positive constant coefficients, {\it Appl. Math. Comput.} {\bf257} (2015), 381--397.

\bibitem{LinXu}
Y. Lin and C. Xu, Finite difference/spectral approximations for the time-fractional diffusion equation, {\it J. Comput. Phys.} {\bf225} (2007), no. 2, 1533--1552.

\bibitem{L17}
Y. Liu, Strong maximum principle for multi-term time-fractional diffusion equations and its application to an inverse source problem, {\it Comput. Math. Appl.} {\bf73} (2017), no. 1, 96--108.

\bibitem{LugKnu}
B. Lug\~ao, D. Knupp and P. Rodriges, Direct and inverse simulation applied to the identification and quantification of point pollution sources in rivers, {\it Environ. Model. Softw.} {\bf156} (2022), Article ID 105488.

\bibitem{Macha}
J. Machado, Fractional derivatives: Probability interpretation and frequency response of rational approximations, {\it Commun. Nonlinear Sci. Numer. Simul.} {\bf14} (2009), 3492--3497.

\bibitem{M25}
M. Machida, The radiative transport equation with waiting time and its diffusion approximation with a time-fractional derivative, {\it J. Comput. Theor. Transp.} (2026), DOI: 10.1080/23324309.2025.2507649.

\bibitem{MagAbd}
R. L. Magin, O. Abdullah, D. Baleanu and X. J. Zhou, Anomalous diffusion expressed through fractional order differential operators in the Bloch--Torrey equation, {\it J. Magn. Reson.} {\bf190} (2008), no. 2, 255--270.

\bibitem{Mai}
F. Mainardi, {\it Fractional Calculus and Waves in Linear Viscoelasticity}, Imperial College, London, 2010.

\bibitem{MaiSpa}
F. Mainardi and G. Spada, Creep, relaxation and viscosity properties for basic fractional models in rheology, {\it Eur. Phys. J. Spec. Topics} {\bf193} (2011), 133--160.

\bibitem{MeeSik}
M. M. Meerschaert and A. Sikorskii, {\it Stochastic Models for Fractional Calculus}, De Gruyter Stud. Math. 43, Walter de Gruyter, Berlin, 2011.

\bibitem{Morozov1984}
V. A. Morozov, {\it Methods for Solving Incorrectly Posed Problems}, Springer, New York, 1984.

\bibitem{Pao}
M. D. Paola, F. Marino and M. Zingales, A generalized model of elastic foundation based on long-range interactions: Integral and fractional model, {\it Int. J. Solids. Struct.} {\bf46} (2009), 3124--3137.

\bibitem{PerKar}
P. Perdikaris and G. Karniadakis, Fractional-order viscoelasticity in one-dimensional blood flow models, {\it Ann. Biomed. Eng.} {\bf42} (2014), 1012--1023.

\bibitem{Pod}
I. Podlubny, {\it Fractional Differential Equations}, Math. Sci. Eng. 198, Academic Press, San Diego, 1999.

\bibitem{Ross}
Y. A. Rossikhin and M. V. Shitikova, Application of fractional calculus for dynamic problems of solid mechanics: Novel trends and recent results, {\it Appl. Mech. Rev.} {\bf63} (2010), Article ID 010801.

\bibitem{SamKil}
S. G. Samko, A. A. Kilbas and O. I. Marichev, {\it Fractional Integrals and Derivatives}, Gordon and Breach Science, Yverdon, 1993.

\bibitem{SchMet}
H. Schiessel, R. Metzler and A. Blumen, Generalized viscoelastic models: Their fractional equations with solutions, {\it J. Phys. A} {\bf28} (1995), no. 23, Article ID 6567.

\bibitem{SchBen}
R. Schumer, D. A. Benson, M. M. Meerschaert and B. Baeumer, Fractal mobile/immobile solute transport, {\it Water Resour. Res.} {\bf39} (2003), no. 10, Article ID 1296.

\bibitem{SunCh}
H. G. Sun, W. Chen, H. Wei and Y. Q. Chen, A comparative study of constant-order and variable-order fractional models in characterizing memory property of systems, {\it Eur. Phys. J. Spec. Top.} {\bf193} (2011), 185--192.

\bibitem{SunZha}
H. Sun, Y. Zhang, D. Baleanu, W. Chen and Y. Chen, A new collection of real world applications of fractional calculus in science and engineering, {\it Commun. Nonlinear Sci. Numer. Simul.} {\bf64} (2018), 213--231.

\bibitem{Sun}
Z.-Z. Sun, {\it Numerical Methods of Partial Differential Equations}, 3rd ed., Science Press, Beijing, 2022.

\bibitem{SunGao}
Z.-Z. Sun and G.-H. Gao, {\it Fractional Differential Equations---Finite Difference Methods}, 2nd ed., De Gruyter, Berlin, 2021.

\bibitem{SunWu}
Z.-Z. Sun and X. Wu, A fully discrete difference scheme for a diffusion-wave system, {\it Appl. Numer. Math.} {\bf56} (2006), no. 2, 193--209.

\bibitem{T26}
E. C. Titchmarsh, The zeros of certain integral functions, {\it Proc. Lond. Math. Soc. (2)} 25 (1926), 283--302.

\bibitem{VanWie}
M. T. Van Genuchten and P. J. Wierenga, Mass transfer studies in sorbing porous media I. Analytical solutions, {\it Soil Sci. Soc. Amer. J.} {\bf40} (1976), no. 4, 473--480.

\bibitem{Vogel2002}
C. R. Vogel, {\it Computational Methods for Inverse Problems}, Front. Appl. Math. 23, Society for Industrial and Applied Mathematics, Philadelphia, 2002.

\bibitem{YeShi}
F. Ye, D. Shi, C. Xu, K. Li, M. Lin and G. Shi, Noninvasive reconstruction of internal heat source in biological tissue using adaptive simulated annealing algorithm, {\it Sci. Rep.} {\bf14} (2024), Article ID 16379.

\bibitem{Zheng2025}
X. Zheng, Two methods addressing variable-exponent fractional initial and boundary value problems and Abel integral equation, {\it CSIAM Trans. Appl. Math.} {\bf6} (2025), no. 4, 666--710.

\bibitem{ZhengLiQiu2024}
X. Zheng, Y. Li and W. Qiu, Local modification of subdiffusion by initial Fickian diffusion: Multiscale modeling, analysis, and computation, {\it Multiscale Model. Simul.} {\bf22} (2024), no. 4, 1534--1557.

\bibitem{ZhengWang2021}
X. Zheng and H. Wang, A hidden-memory variable-order time-fractional optimal control model: Analysis and approximation, {\it SIAM J. Control Optim.} {\bf59} (2021), no. 3, 1851--1880.

\bibitem{ZhengWang2021b}
X. Zheng and H. Wang, Optimal-order error estimates of finite element approximations to variable-order time-fractional diffusion equations without regularity assumptions of the true solutions, {\it IMA J. Numer. Anal.} {\bf41} (2021), no. 2, 1522--1545.

\end{thebibliography}
\end{document}